\documentclass[11pt]{article}
\usepackage{amsfonts}
\usepackage{amsthm,amsmath}
\usepackage{graphicx}
\usepackage{latexsym}
\makeatletter
\renewcommand{\@seccntformat}[1]
{\csname the#1\endcsname.\enspace}
\makeatother
\usepackage{color}
\theoremstyle{definition}

\newtheorem{conjecture}{Conjecture}
\newtheorem{lemma}{Lemma}

\newtheorem{remark}{Remark}
\newtheorem{example}{Example}

\newcommand{\K}{\text{K}}

\newcommand{\Y}{\textbf{Y}}

\newcommand{\X}{\textbf{X}}

\newcommand{\x}{\textbf{x}}
\newcommand{\EE}{\mathbb{E}}

\newcommand{\PP}{\mathbb{P}}

\begin{document}

\title{\textbf{\sc On Uncertainty and  Information Properties Of   Ranked Set Samples}}
\author{Mohammad Jafari Jozani${}^{a,}$\footnote{Corresponding author.
 \newline
{E-mail addresses:} m$_{-}$jafari$_{-}$jozani@umanitoba.ca, (M. Jafari Jozani), ahmadi-j@um.ac.ir (J. Ahmadi).
}~~
 and  Jafar  Ahmadi${}^{b}$\\\\
 \footnotesize{${}^a$University of Manitoba, Department of Statistics, Winnipeg, MB, CANADA, R3T 2N2 } \\\footnotesize{${}^{b}$Department of Statistics,  Ferdowsi University of Mashhad,   P.O. Box 91775-1159,  Mashhad,\ Iran }
  }
\date{}
\maketitle
\vspace{-1cm}
\begin{abstract}
Ranked set sampling is a sampling design which has a wide range of applications in industrial statistics,  and environmental and ecological studies, etc.. It  is well known that ranked set samples provide more Fisher information than simple random samples of the same size  about the unknown parameters of the underlying distribution in parametric inferences.  In this paper, we consider the uncertainty and  information content of ranked set samples in both perfect and imperfect ranking scenarios
in terms of Shannon entropy, R\'enyi  and Kullback-Leibler (KL) information measures.
It is proved that under these information measures, ranked set sampling design performs better than  its simple random sampling  counterpart of the same size. The information content is also a monotone function of the set size in ranked set sampling. Moreover, the effect of ranking error on the information content of the data is  investigated. 
\end{abstract}

\section{Introduction and Preliminaries }
During the past few years, ranked
set sampling  has  emerged  as a powerful tool in statistical inference, and it is now regarded as a serious alternative
to the commonly used simple random sampling  design. Ranked set sampling and some of its variants  have been applied successfully in different areas of applications such as  industrial statistics, environmental and ecological studies, biostatistics and statistical genetics.
 The feature of ranked set sampling  is that it combines simple random sampling with other  sources of information  such as professional knowledge,   auxiliary information, judgement,  etc., which are assumed to be inexpensive and easily obtained. This extra information helps to  increase the chance that the collected sample  yields more representative measurements (i.e., measurements that span the range of the value of the variable of interest in the underlying population).   In its original form, ranked set sampling  involves randomly drawing $k$ units (called a set of size $k$) from the underlying population for which an estimate of the unknown parameter of interest  is required. The units of this set are ranked by means of an auxiliary variable or some other ranking process such as judgmental ranking.  For this ranked set,  the unit ranked lowest is chosen for actual measurement of the variable of interest. A second set of size $k$ is then  drawn and ranking carried out. The unit in the second lowest position is chosen and the variable of interest for this unit is quantified. Sampling is continued until, from the $k$th set, the $k$th  ranked unit is measured. This entire process may be repeated $m$ times (or cycles) to obtain a ranked set sample of size $n=mk$ from the underlying population.

Let $\X_{SRS}=\{ X_i, i=1, \ldots, n\}$ be a simple random sample (SRS) of size $n\geq 1$ from  a continuous  distribution with probability distribution function (pdf) $f(x)$. Let $F(x)$  denote the cumulative distribution function (cdf) of the random variable $X$  and define  $\bar{F}(x)= 1-F(x)$ as the survival function of $X$  with support $S_X$. Also assume that $\X_{RSS}=\{ X_{(i)j}, i=1, \ldots, k, j=1, ..., m\}$ denotes a ranked set sample  (RSS) of size $n=mk$ from $f(x)$ where $k$ is the set size and $m$ is the cycle size. Here  $X_{(i)j}$ is the $i$th order statistic in a set of size $k$ obtained in cycle $j$ with pdf
 $$f_{(i)}(x)= \frac{k!}{(i-1)! (k-i)!} F^{(i-1)}(x) \bar{F}^{(k-i)}(x) f(x),  \  \  \   x\in S_X.$$  
 
When ranking is imperfect we use $\X^*_{RSS}=\{ X_{[i]j}, i=1, \ldots, k, j=1, ..., m\}$ to denote an imperfect RSS of size $n=mk$ from $f(x)$. We also use $f_{[i]}(x)$ to show the pdf of the judgemental  order statistic  $X_{[i]}$ which is given by 
\begin{equation}
\label{imperfect-pdf}
f_{[i]}(x) = \sum_{r=1}^n p_{i,r} f_{(r)}(x),
\end{equation}
 where $p_{i,r}=\PP(X_{[i]}= X_{(r)} )$ denotes the probability with which the $r$th order statistic is judged as having rank $i$ with $\sum_{i=1}^kp_{i,r}= \sum_{r=1}^k p_{i, r}=1$. 
Readers are referred to Wolfe (2004, 2010), Chen et al.\ (2004) and references therein for further details.

The Fisher information plays a central role in statistical inference and  information theoretic studies.  It  is well known that RSS  provides more Fisher information than SRS  of the same size  about the unknown parameters of the underlying distribution in parametric inferences (e.g., Chen, 2000,   Chapter 3).    Park and Lim (2012) studied the effect of imperfect rankings on the amount of Fisher information in ranked set samples.
 Frey (2013) showed  by example that the Fisher information in an imperfect ranked set sample may be higher than the Fisher information in a perfect ranked-set sample.
  The concept of information is so rich that there is no single definition that will be able to quantify the information content of a sample properly. For example, from an engineering perspective,  the Shannon entropy or the R\'enyi information might be more suitable to be used  as  measures to quantify the information content of a sample than the Fisher information.  
In this paper, we study the notions of uncertainty and  information content of RSS data in both perfect and imperfect ranking scenarios under  the Shannon entropy, R\'enyi  and Kullback-Leibler (KL) information measures and compare them  with their counterparts with   SRS data. These measures are  increasingly being used in various contexts such as order statistics by Wong and Chen (1990) and Park (1995),  Ebrahimi et al. (2004), Bratpour et al. (2007a, b), censored data by    Abo-Eleneen,  (2011), record data  and reliability and life testing context by  Raqab and Awad (2000, 2001), Zahedi and Shakil (2006),   Ahmadi and Fashandi (2008) and in testing hypothesis by Park (2005), Balakrishnan  et al.  (2007)  and Habibi Rad et al.  (2011). So,  it would be of interest to use these measures to calculate  the information  content of RSS data and compare  them with  their counterparts  with   SRS  data. 

To this end, in Section \ref{sec-shannon}, we obtain  the Shannon entropies of RSS and SRS data of the same size. We show that the difference between the Shannon entropy of $\X_{RSS}$ and $\X_{SRS}$ is distribution free and it is a monotone function of the set size in ranked set sampling.  In Section \ref{sec-renyi}, similar results are obtained under the R\'enyi  information. 
Section \ref{sec-kl} is devoted to the Kullback-Leibler information of RSS data and its comparison with its counterpart under SRS data.   We show  that the Kullback-Leibler information  between
the distribution of  $\X_{SRS}$ and  distribution of $\X_{RSS}$ is distribution-free and  increases as the set size increases. Finally,  in Section \ref{sec-concluding},  we provide some concluding  remarks.

%
%

\section{Shannon Entropy of Ranked Set Samples } \label{sec-shannon}
The Shannon entropy or simply the entropy of a continuous random variable $X$   is defined by 
\begin{align}
\label{e1}
   {\rm H}(X)=-\int f(x)\log f(x) \, dx,
\end{align}
provided the integral exists. 
The Shannon entropy is extensively used in the  literature as a
quantitative measure of uncertainty associated with a random
phenomena. The development of the idea of the entropy by Shannon (1948) initiated  a separate branch of learning named the ``Theory of Information".  
 The Shannon entropy  provides an excellent tool to quantify the amount of information (or uncertainty) contained in a sample regarding its parent distribution.
Indeed, the amount of information which we get when we observe the result on a random experiment can be taken to be equal to the amount of uncertainty concerning the outcome of the experiment before carrying it out.   In practice, smaller values of the Shannon entropy are more desirable.  We refer the reader to Cover and Thomas (1991) an  references therein for more details.   In this section, we compare the  Shannon entropy  of SRS data with  its counterparts under both perfect and imperfect RSS data  of the same size.  Without loss of generality\textcolor[rgb]{1,0,0}{,} we take $m=1$ throughout the paper.  From \eqref{e1}, the Shannon  entropy of $\X_{SRS}$ is given by
\begin{eqnarray*}
{\rm H}(\X_{SRS}) = -\sum_{i=1}^n \int f(x_i) \log f(x_i)\, dx_i= n {\rm H}(X_1).
\end{eqnarray*}
\noindent Under the perfect ranking assumption, it is easy to see that
\begin{eqnarray}\label{H1}
{\rm H} (\X_{RSS})= -\sum_{i=1}^n \int f_{(i)} (x) \log f_{(i)}(x)\, dx = \sum_{i=1}^n H(X_{(i)}),
\end{eqnarray}
where ${\rm H} (X_{(i)})$ is the entropy of the $i$th order statistic in a sample of size $n$.  Ebrahimi et al. (2004) explored some properties of the  Shannon entropy of the usual  order statistics (see also, Park, 1995; Wong and Chen, 1990).  Using \eqref{e1} and the transformation   $X_{(i)}=F^{-1}(U_{(i)})$ it is easy to prove the following 
   representations for  the Shannon  entropy of order statistics (see,  Ebrahimi et al. 2004,  page 177):
   \begin{equation}
   \label{eo-1}
  {\rm  H}(X_{(i)})={\rm H}(U_{(i)})- \EE\left[\log {[f(F^{-1}(W_i))]}\right],
   \end{equation}
where  $W_i$ has the beta distribution with parameters $i$ and $n-i+1$ and  $U_{(i)}$ stands for the $i$th order statistic  of  a random sample of  size $n$ from the Uniform$(0, 1)$  distribution.

\noindent In the following result, we show that the Shannon entropy  of RSS data is smaller than its SRS counterpart when ranking is perfect. 

\begin{lemma}\label{H(XRSS)} ${\rm H}(\X_{RSS})\leq {\rm H}(\X_{SRS})$ for all set size $n\in\mathbb{N}$ and the equality  holds when $n=1$.
\end{lemma}
\begin{proof}
To show the result we use  the fact that  $f(x)=\frac 1n \sum_{i=1}^n f_{(i)}(x)$ (see Chen et al., 2004).
Using the convexity of $g(t)=t\log t$ as a function of $t>0$,  we have 
\begin{eqnarray}\label{H2}\frac {1}{n}   \sum_{i=1}^n f_{(i)}(x) \log f_{(i)}(x)  \geq  \left( \frac{1}{n} \sum_{i=1}^n f_{(i)}(x)\right) \left(\log  \frac{1}{n} \sum_{i=1}^n f_{(i)}(x) \right) = f(x) \log f(x).\end{eqnarray}
Now,  the result follows  by the use of  \eqref{H1} and \eqref{H2}.
\end{proof}

\vskip 3mm

\noindent In the sequel, we  quantify the difference between ${\rm H}(\X_{RSS})$ and ${\rm H}(\X_{SRS})$. To this end, by \eqref{eo-1}, we first  get
\begin{align*}
{\rm H}(\X_{RSS})
&= \sum_{i=1}^n  {\rm H}(U_{(i)}) - \int \sum_{i=1}^n  f_{(i)}(x) \log f(x) dx \\
&=  \sum_{i=1}^ n {\rm H}(U_{(i)})  + {\rm H}(\X_{SRS}).
\end{align*}
Note that since  ${\rm H}(\X_{RSS})\leq {\rm H}(\X_{SRS})$ we must have $\sum_{i=1}^n {\rm H}(U_{(i)})\leq 0$, for all $n\in\mathbb{N}$. Also, ${\rm H}(\X_{RSS})-{\rm H}(\X_{SRS})=\sum_{i=1}^n {\rm H}(U{(i)})$ is distribution-free (doesn't depend on the parent distribution). Ebrahimi et al. (2004) obtained an expression for  ${\rm H}(U_{(i)})$ which is given by
\begin{equation*}
{\rm H}(U_{(i)})=\log B(i, n-i+ 1)- (i-1)[\psi(i)-\psi(n + 1)]- (n-i)[\psi(n-i+ 1)-\psi(n + 1)],
\end{equation*}
where $\psi(z)=\frac{d}{dz}\log \Gamma(z)$ is the digamma function and $B( a, b)$ stands for the complete beta function. Hence, we have 
\begin{align*}
{\rm H}(\X_{RSS})-{\rm H}(\X_{SRS})
&=2\sum_{j=1}^{n-1}(n-2j)\log j -n \log n-  2\sum_{i=1}^n(i-1)\psi(i)+n(n-1)\psi(n + 1)\\
&=k(n), \ \ \mbox{say}.
\end{align*}
By noting that $\psi(n+1)=\psi(n)+1/n$,  for $n\geq 2$, we can easily find the following recursive formula for calculating $k(n)$:
$$k(n+1)=k(n) + n + \log \Gamma(n)-(n+1)\log(n+1).$$

\noindent  Table \ref{shannon} shows   the numerical  values of ${\rm H}(\X_{RSS})-{\rm H}(\X_{SRS})$ for $n\in\{2, \ldots,  10\}$.  From Table \ref{shannon},  it is observed   that the difference  between  the Shannon entropy of RSS data  and  its SRS counterpart increases as  the set size increases.
 However, intuitively, this can be explained by the fact that ranked set sampling provides more structure to the observed data  than simple random sampling  and the amount of the uncertainty in the more structured RSS data set is less than  that of SRS. 

\begin{table}[htdp]
\caption{ \footnotesize{The numerical values of $k(n)$  for  $n=2$ up to 10.} }
\begin{center}
\begin{tabular}{c|ccccccccccc}
 $n$&2&3&4&5&6&7&8&9&10\\
    \hline
$k(n)$ &-0.386 & -0.989 & -1.742& -2.611& -3.574& -4.616& -5.727& -6.897& -8.121\\
\end{tabular}
\end{center}
\label{shannon}
\end{table}%

\noindent Now, assume that $\X^*_{RSS}=\{ X_{[i]}, i=1, \ldots, n\} $ is  an imperfect RSS of size $n$ from $f(x)$.
Similar to the  perfect  RSS  we  can easily show that 
\begin{align}\label{H1-imp}
{\rm H}( \X^*_{RSS}) = \sum_{i=1}^n {\rm H}( X_{[i]}), 
\end{align}
 where we assume that the cycle size is equal to one and $k=n$. Also ${\rm H}(X_{[i]}) = -\int f_{[i]}(x) \log f_{[i]}(x)\, dx$,  or equivalently 
 $$ {\rm H}(X_{[i]}) = -\int \left( \sum_{r=1}^n p_{i,r} f_{(r)}(x) \right)\log \left( \sum_{r=1}^n p_{i,r}f_{(r)}(x) \right) dx.$$

\noindent Again, using the convexity of $g(t)=t\log t$ and the equalities  $\sum_{r=1}^n p_{i,r}= \sum_{i=1}^n p_{i,r}=1$, we find
\begin{align*}
{\rm H}({\X^*_{RSS} } )
&=\sum_{i=1}^n {\rm H}(X_{[i]})\\ 
&\leq -n \int  \left(\frac1n \sum_{r=1}^n (\sum_{i=1}^n p_{i,r}) f_{(r)}(x) \right)  \log   \left(\frac1n \sum_{r=1}^n (\sum_{i=1}^n p_{i,r}) f_{(r)}(x)\right) dx\\
&= - n \int f(x) \log f(x) dx\\
&={\rm H}(\X_{SRS}).
\end{align*}
So, we have the following result.

\begin{lemma}
${\rm H}(\X^*_{RSS})\leq {\rm H}(\X_{SRS})$ for all set size $n\in\mathbb{N}$ and the equality holds  when the ranking is done randomly and $p_{i,r}=\frac 1n$, for all $i, r\in\{ 1, \ldots, n\}$.
\end{lemma}

In the following result we compare the  Shannon entropies of  perfect and imperfect RSS data. We observe that the Shannon entropy of $\X_{RSS}$ is less than the Shannon entropy of $\X^*_{RSS}$.

\begin{lemma} ${\rm H}(\X_{RSS})\leq {\rm H}(\X^*_{RSS})$ for all set size $n\in\mathbb{N}$ and the equality happens when the ranking is perfect.
\end{lemma}
\begin{proof}
Using  the inequality
$ f_{[i]}(x)\, \log f_{[i]}(x) \leq \sum_{r=1}^n p_{i,r} f_{(r)}(x) \log f_{(r)}(x) $, we have 
\begin{align*}
{\rm H}(X_{[i]})
\geq -\sum_{r=1}^n p_{i,r} \int f_{(r)}(x) \log f_{(r)}(x) dx= \sum_{r=1}^n p_{i, r} H(X_{(r)}).
\end{align*}
Now, the result follows from  \eqref{H1-imp} upon changing the order of summations and using  $\sum_{i=1}^np_{i, r}=1$.
\end{proof}

Summing up, we find the following ordering relationship among the Shannon  entropies  of $\X^*_{RSS}$, $\X_{RSS}$ and $\X_{SRS}$:
$${\rm H}(\X_{RSS}) \leq {\rm H}(\X^*_{RSS})\leq {\rm H}(\X_{SRS}).$$

\begin{example}\label{ex-exp}

Suppose $X$ has  an exponential  distribution with pdf  $f(x)=\lambda e^{-\lambda x}$,  $x>0$, where  $\lambda>0$ is the unknown parameter of interest. We consider the case  where $n=2$. For an imperfect RSS of size $n=2$, we use the ranking error probability matrix 
$$P =\left[ \begin{matrix} p_{1,1}&  p_{1, 2}\\ p_{2, 1} &p_{2,2}  \end{matrix}\right].$$
 Using  \eqref{imperfect-pdf},  we have  $f_{[i]}(x)=2\lambda e^{- \lambda x} \left[(p_{i,1}-p_{i,2}) e^{- \lambda x}+ p_{i,2}\right]$, $i=1, 2$. 
Straightforward calculations show that  ${\rm H}({\bf X}_{SRS})=2-2\log\lambda$,  ${\rm H}({\bf X}_{RSS})=3- 2\log (2\lambda)$, and  
\begin{align*}
{\rm H}(X^*_{RSS})&=2- 2\log (2\lambda)+ (p_{2,2}-p_{1,1})+ \eta(p_{1,1})+\eta(p_{2,2}),  
\end{align*} 
where
\begin{eqnarray*}
\eta (a)=\frac{2}{1-2a} \int_{a}^{1-a} u \log u \ du
=\frac{1}{2}+\frac{1}{1-2a}\left[ (1-a)^2\log(1-a)-a^2\log a\right],
\end{eqnarray*}
 with  $0<a<1$. It is easy to show that 
\begin{align*}
{\rm H}({\bf X}_{RSS})- {\rm H}({\bf X}_{SRS})  &= 1-2\log 2\approx -0.3863<0,\\
 {\rm H}({\bf X}^*_{RSS})- {\rm H}({\bf X}_{SRS})&=\eta (p_{1,2})+\eta (1-p_{1,2})=2\eta (p_{1,2})-2 \log 2<0,\\
{\rm H}({\bf X}_{RSS})-{\rm H}({\bf X}^*_{RSS})&=1-\eta (p_{1,2})-\eta (1-p_{1,2})=1-2\eta (p_{1,2})<0.
\end{align*} 
Figure \ref{fig-infoimprovement} shows the differences between ${\rm H}(\X^*_{RSS})$ and ${\rm H}(\X_{RSS})$ with ${\rm H}(\X_{SRS})$. It also presents the effect of ranking error on the amount of the Shannon entropy of the resulting  RSS data by comparing ${\rm H}(\X^*_{RSS})$ with ${\rm H}({\bf X}_{RSS})$.
 It is observed that,  the maximum difference occurs for $p_{1,2}=0.5$.

\begin{figure}[h] 
    \centering
    \includegraphics[width=4.5in, height=3.5in]{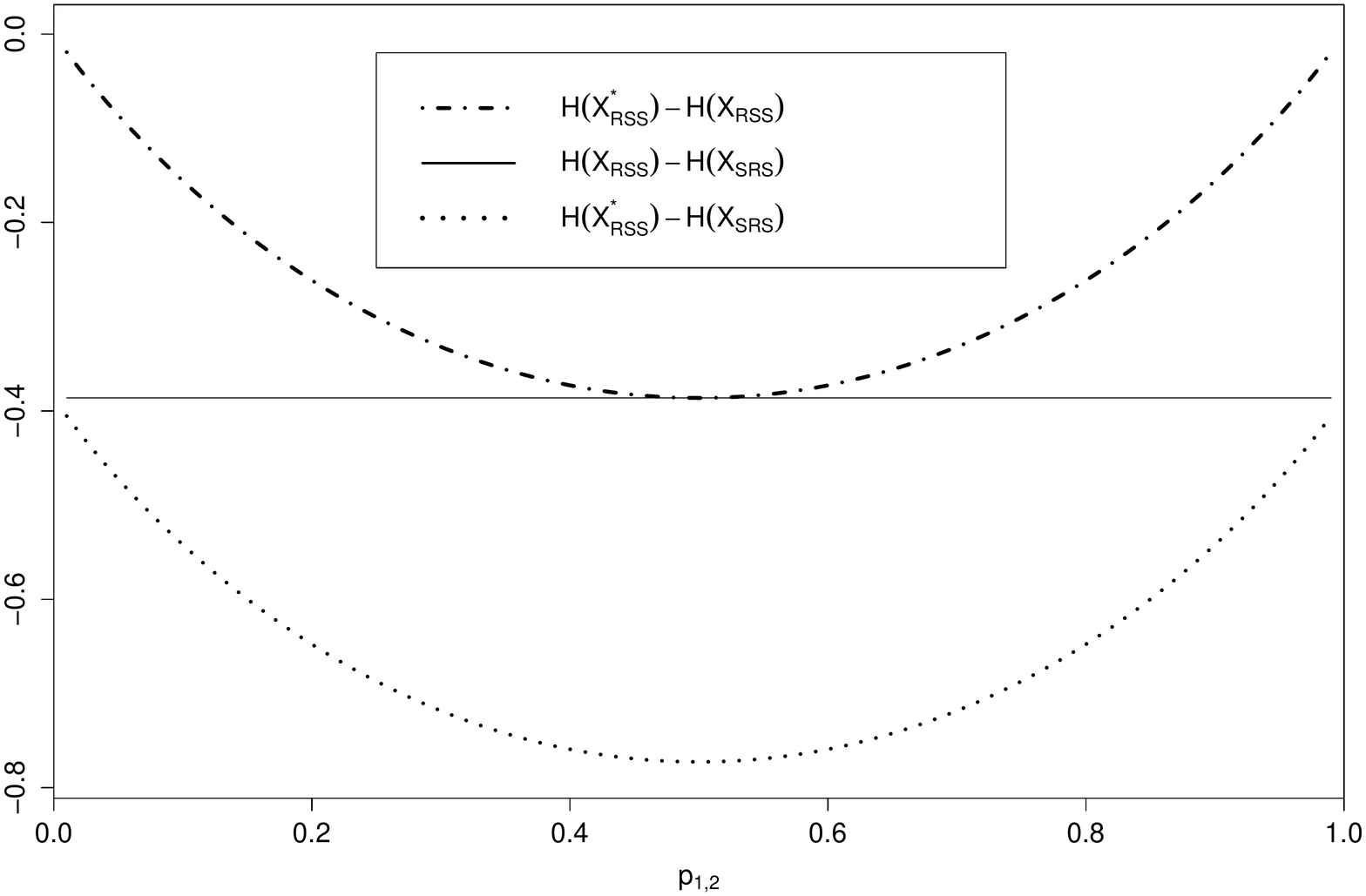}
        \caption{\footnotesize{Computed values of the difference between the Shannon entropies of ${\bf X}_{RSS}$ and ${\bf X}^*_{RSS}$ compared with   that of ${\bf X}_{SRS}$ of the same size as a function of the ranking error probability $p_{1, 2}$}.}\label{fig-infoimprovement}
    \end{figure}
    \end{example}

\section{ R\'enyi Information of Ranked Set Samples}\label{sec-renyi}
A more general measure of entropy with the same meaning and similar properties as that of Shannon entropy has been defined by R\'enyi (1961) as follows
\begin{align}
\label{Reni}
{\rm H}_{\alpha}(X)
= \frac{1}{1-\alpha} \log \int f^{\alpha}(x) d\nu(x)
= \frac{1}{1-\alpha} \log \EE\left[ f^{\alpha-1}(X)\right],
\end{align}
where $ \alpha>0$, $\alpha \neq 1$ and $d\nu(x)=dx$ for the continuous and $d\nu(x)=1$ for discrete cases. It is well known that 
$$\lim_{\alpha\to 1}{\rm H}_{\alpha}(X)= -\int f(x) \log f(x) dx={\rm  H}(X).$$
R\'enyi information  is much more flexible than the Shannon entropy  due to
the parameter $\alpha$.  It is an  important  measure   in  various applied sciences such as statistics, ecology,  engineering, economics, etc.  
In this section,  we obtain the R\'enyi information of $\X_{RSS}$ and $\X^*_{RSS}$ and compare them with the R\'enyi information of $\X_{SRS}$. To this end, from \eqref{Reni},  it is easy to show that the R\'enyi information of  a SRS of size $n$ from $f$  is given by 
\begin{equation}
\label{Rsrs1}
{\rm H}_{\alpha}(\X_{SRS})= \sum_{i=1}^n {\rm H}_{\alpha}(X_i) = n {\rm H}_{\alpha}(X_1).
\end{equation}
Also, for a RSS  of  size $n$,  we have
\begin{equation}
\label{Rrss}
 {\rm H}_{\alpha}(\X_{RSS})= \sum_{i=1}^n {\rm H}_{\alpha}(X_{(i)}).
\end{equation}
To compare ${\rm H}_{\alpha}(\X_{SRS})$ with  ${\rm H}_{\alpha}(\X_{RSS})$ and ${\rm H}_{\alpha}(\X^*_{RSS})$, we consider two cases,
i.e.   $0<\alpha<1$ and  $\alpha>1$.  First, we find the results for $0<\alpha<1$  which  are stated in the next lemma.

\begin{lemma}\label{renyi<1}
For any $0<\alpha<1$ and all $n\in\mathbb{N}$, we have 
$${\rm H}_{\alpha}(\X_{RSS}) \leq {\rm H}_{\alpha}(\X^*_{RSS})\leq {\rm H}_{\alpha}(\X_{SRS}).$$
\end{lemma}
\begin{proof}
We first show that for any $0<\alpha<1$,  ${\rm H}_{\alpha}(\X_{RSS}) \leq {\rm H}_{\alpha}(\X^*_{RSS})$. To this end, using 
\begin{align}
\label{imperfect-1} 
{\rm H}_{\alpha}(\X^*_{RSS})=\frac{1}{1-\alpha} \sum_{i=1}^n \log \int \left( \sum_{j=1}^n p_{i,j}f_{(j)}(x)\right)^{\alpha} dx,
\end{align}  
and   concavity of  $h_1(t)=t^{\alpha}$,  for $0<\alpha<1$, $t>0$,   we have
\begin{eqnarray*}
{\rm H}_{\alpha}(\X^*_{RSS})&\geq & \frac{1}{1-\alpha} \sum_{i=1}^n \log \int \sum_{j=1}^n p_{i,j}  f_{(j)}^{\alpha}(x)\,  dx\\
&\geq &  \frac{1}{1-\alpha} \sum_{i=1}^n  \sum_{j=1}^n p_{i,j} \log \int f_{(j)}^{\alpha}(x)\,  dx\\
&=& \frac{1}{1-\alpha}   \sum_{j=1}^n  \log \int  f_{(j)}^{\alpha}(x)\, dx
=  {\rm H}_{\alpha}(\X_{RSS}),
\end{eqnarray*}
where the second inequality is obtained by using  the concavity of   $h_2(t)=\log t$, for $t>0$. This,  with  \eqref{Rsrs1},  shows the result. To complete the proof  we show that ${\rm H}_{\alpha}(\X^*_{RSS}) \leq {\rm H}_{\alpha}(\X_{SRS})$ for any $0<\alpha<1$ and all $n\in\mathbb{N}$. 
To this end, from  \eqref{imperfect-1},  and using $f(x)= \frac{1}{n}\sum_{i=1}^n f_{[i]}(x)$,  we have 
\begin{eqnarray*}
\label{imperfect-2} \nonumber
{\rm H}_{\alpha}(\X^*_{RSS})&=& \frac{1}{1-\alpha} \sum_{i=1}^n \log \int f_{[i]}^{\alpha}(x)\,  dx\\  \nonumber
&\leq & \frac{n}{1-\alpha} \log  \sum_{i=1}^n \frac{1}{n} \int f_{[i]}^{\alpha}(x)\, dx\\  \nonumber
&\leq & \frac{n}{1-\alpha} \log \int  \left( \frac{1}{n}\sum_{i=1}^n  f_{[i]}(x)\right)^{\alpha} dx\\  \nonumber
&=&  \frac{n}{1-\alpha} \log \int    f^{\alpha}(x) dx
= {\rm H}(\X_{SRS}).
\end{eqnarray*}  
\end{proof}

\noindent 
In Lemma \ref{renyi<1}\textcolor[rgb]{1,0,0}{,} we were able to show analytically an ordering relationship among the R\'enyi information of $\X^*_{RSS}$, $\X_{RSS}$ and $\X_{SRS}$ when $0<\alpha<1$. It would naturally be of interest to extend such a relationship to the case where $\alpha>1$. 
 It appears  that similar relationship as in Lemma \ref{renyi<1} holds when $\alpha>1$.  However, we have not analytical  proof here.

\begin{conjecture} For any $\alpha>1$ and all $n\in\mathbb{N}$, we have 
${\rm H}_{\alpha}(\X_{RSS}) \leq {\rm H}_{\alpha}(\X^*_{RSS})\leq {\rm H}_{\alpha}(\X_{SRS}).$

\end{conjecture} 

In Example \ref{ex-exp-cont} we compare the R\'enyi information of $\X^*_{RSS}$, $\X_{RSS}$
 and $\X_{SRS}$  as a function of $\alpha$ in the case of  an exponential distribution. The results are presented in Figure \ref{fig-renyi}, which do support Conjecture 1.

\begin{example}
\label{ex-exp-cont}
 Suppose the assumptions of   Example \ref{ex-exp} hold, then the R\'enyi information of a SRS of size $n=2$ is given by
$${\rm H}_{\alpha}(\X_{SRS})= - 2\log \lambda - \frac{2}{1-\alpha}\log \alpha, \quad \alpha\neq 1. $$
Straightforward calculations show that 
\begin{align*}
{\rm H}_{\alpha}(X_{(1)})
&= -\log \lambda - \log 2 -\frac{1}{1-\alpha}\log \alpha,\\
{\rm H}_{\alpha}(X_{(2)}) &= -\log\lambda + \frac{\alpha}{1-\alpha}\log 2 +\frac{1}{1-\alpha} \log\left\{ \frac{\Gamma(\alpha+1) \Gamma(\alpha)}{\Gamma(2\alpha+1)}\right\},
\end{align*}
and so the R\'enyi information of $\X_{RSS}$ is given by ${\rm H}_{\alpha}(\X_{RSS})= {\rm H}_{\alpha}(X_{(1)})+ {\rm H}_{\alpha}(X_{(2)})$. Now, 
$${\rm H}_{\alpha}(\X_{RSS})-{\rm H}_{\alpha}(\X_{SRS}) = \frac{\alpha}{1-\alpha} (1-\log 2) + \frac{1}{1-\alpha} \log\left\{ \frac{\Gamma(\alpha+1) \Gamma(\alpha)}{\Gamma(2\alpha+1)}\right\}. $$
To obtain ${\rm H}_{\alpha}({\X^*_{RSS}})$, let 
$$U_{i, \lambda}(x, t)= a_i(t) e^{-\lambda\, x} + b_i(t), \quad i=1, 2 ,$$
where 
$$a_i(t)= (-1)^i (1-2t) \quad\mbox{and}\quad b_i(t)= t^{(1-i)} (1-t)^{(2-i)},$$
and $p_{1,1}= P(X_{(1)}= X_{[1]})$ is defined in Example  \ref{ex-exp}.  Now, the R\'enyi information of $\X^*_{RSS}$ is obtained as follows 
$$ {\rm H}_{\alpha}(\X^*_{RSS})= \frac{\alpha}{1-\alpha}\log 2\lambda + \frac{1}{1-\alpha}\sum_{i=1}^2\log \int_0^{\infty} e^{-\alpha \lambda\, x} U^{\alpha}_{i, \lambda}(x, p_{1,1})\,  dx,\quad \alpha\neq 1, $$
which can be calculated numerically.  Figure \ref{fig-renyi}(a)  shows the values of $H_{\alpha}(\X^*_{RSS}) - {\rm H}_{\alpha}(\X_{SRS})$  as a function of $\alpha$ for  $p_{1, 1}\in\{ 0.8, 0.9, 0.95, 1\}$. When $p_{1, 1}=1$, $H_{\alpha}(\X^*_{RSS}) - {\rm H}_{\alpha}(\X_{SRS})=H_{\alpha}(\X_{RSS}) - {\rm H}_{\alpha}(\X_{SRS}). $ In   Figure \ref{fig-renyi}(b)  we show the effect of the  ranking error on the R\'enyi information of $\X_{RSS}$ by comparing ${\rm H}_{\alpha}(\X^*_{RSS})$ and ${\rm H}_{\alpha}(\X_{RSS})$ as functions of $\alpha$ for different values of $p_{1, 1}$.
\begin{figure}[h!] 
    \centering
    \includegraphics[width=7.2in, height=3in]{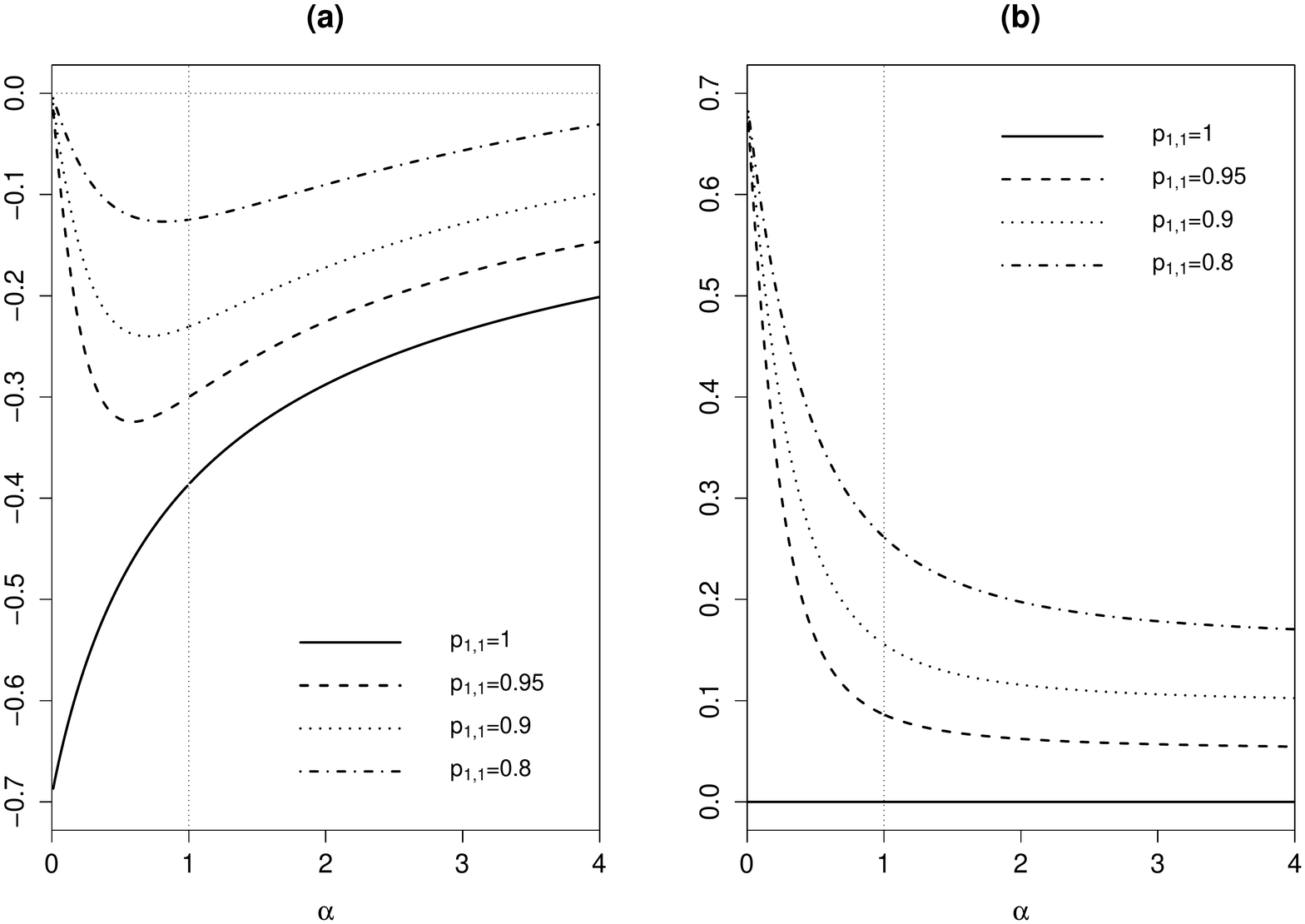}
        \caption{\footnotesize{Comparison of  the R\'enyi information  of ${\bf X}_{RSS}$ and ${\bf X}^*_{RSS}$  with   that of ${\bf X}_{SRS}$  as a function of $\alpha$.  The value of  $H_{\alpha}(\X^*_{RSS}) - {\rm H}_{\alpha}(\X_{SRS})$ are presented in (a) while $H_{\alpha}(\X^*_{RSS}) - {\rm H}_{\alpha}(\X_{RSS})$ are given in (b).   }}\label{fig-renyi}
    \end{figure}

\end{example}
 
  Note that for  $\alpha>1$   the difference between the  R\'enyi information of ${\bf X}_{RSS}$ with  its counterpart under SRS can be written as follows
\begin{align}
\label{a>1-2} \nonumber
{\rm H}_{\alpha}(\X_{RSS}) - {\rm H}_{\alpha}(\X_{SRS})
&=   \frac{1}{1-\alpha} \sum_{i=1}^n \log \int f^{\alpha}_{(i)}(x) dx - \frac{n}{1-\alpha} \log \int  f^{\alpha} (x)dx\\ \nonumber
&=  \frac{1}{1-\alpha} \sum_{i=1}^n \log\left ( \frac{\int f^{\alpha}_{(i)}(x) dx}{ \int  f^{\alpha} (x)dx} \right)\\
&= \frac{\alpha}{1-\alpha} n \log n+    \frac{1}{1-\alpha} \sum_{i=1}^n \log \EE\left[ \{\PP_{F(W)}( T= i-1) \}^{\alpha} \right],
\end{align}
where $T|W=w\sim Bin(n-1, F(w))$ and $W$ has a density proportional to $f^{\alpha}(w)$, i.e. $g(w) = \frac{f^{\alpha}(w)}{\int f^{\alpha}(w)dw}$.   
Since  $\PP_{F(w)}( T= i-1) \leq 1$ for all $i=1, \ldots, n-1$ and fixed $w$, we have 
$$ \log \EE\left[ \{\PP_{F(W)}( T= i-1) \}^{\alpha}\right] \leq 0,$$ for all $\alpha>1$. This results in  a lower bound for the difference between the R\'enyi information of $\X_{RSS}$ and $\X_{SRS}$ as  ${\rm H}_{\alpha}(\X_{RSS}) - {\rm H}_{\alpha}(\X_{SRS}) \geq \frac{\alpha}{1-\alpha} n \log n.$
In the following result, we  find a sharper lower bound for ${\rm H}_{\alpha}(\X_{RSS}) - {\rm H}_{\alpha}(\X_{SRS})$  when $\alpha>1$.
 
\begin{lemma}
For any $\alpha>1$ and all $n\geq 2$,  we have 
${\rm H}_{\alpha}(\X_{RSS}) -{\rm H}_{\alpha}(\X_{SRS}) \geq  \Psi(\alpha, n)$, with 
$$\Psi(\alpha, n)=  \frac{\alpha}{1-\alpha}\sum_{i=1}^n \log \left\{ n \binom{n-1}{i-1}\left (\frac{i-1}{n-1}\right)^{i-1} \left(\frac{n-i}{n-1}\right)^{n-i} \right\}, $$
where $ \Psi(\alpha, n)\in\left[\frac{n\alpha}{1-\alpha}\log n, 0  \right).$
\end{lemma}
\begin{proof}
Using \eqref{Rrss}, the pdf of $X_{(i)}$  and the transformation $F(X)=U$,  we have 
\begin{eqnarray*}
{\rm H}_{\alpha}(\X_{RSS})
 &=&\frac{1}{1-\alpha} \sum_{i=1}^n \log \int_0^1 \left\{f^*_{i, n-i+1}(u)\right\}^{\alpha} f^{\alpha-1}(F^{-1}(u)) du,
\end{eqnarray*}
where $f^*_{i, n-i+1}(u)$ is the pdf of a $Beta(i, n-i+1)$ random variable with its mode at $u^*=\frac{i-1}{n-1}$. Now, since $f^*_{i, n-i+1}(u) \leq f^{*}_{i, n-i+1}(\frac{i-1}{n-1})$,  we have
  \begin{eqnarray*}
  {\rm H}_{\alpha}(\X_{RSS})&\geq  &  \frac{\alpha}{1-\alpha}\sum_{i=1}^n \log \left\{ n \binom{n-1}{i-1} (\frac{i-1}{n-1})^{i-1} (\frac{n-i}{n-1})^{n-i} \right\} +  \frac{n}{1-\alpha} \log\int_0^1 f^{\alpha-1}(F^{-1}(u)) du 
  \\&=& \Psi(\alpha, n) +  {\rm H}({\X_{SRS}}),
  \end{eqnarray*}
  where   
  \begin{eqnarray}
  \Psi(\alpha, n)=  \frac{\alpha}{1-\alpha}\sum_{i=1}^n \log \left\{ n \binom{n-1}{i-1} (\frac{i-1}{n-1})^{i-1} (\frac{n-i}{n-1})^{n-i} \right\}.
  \end{eqnarray}
  It is easy  to show that for all $n\in\mathbb{N}$ and $\alpha>1$,  $\Psi(\alpha, n)<0$. To do this,  one can easily check that $\Psi(\alpha, n+1)\leq \Psi(\alpha, n)$ for all $n\geq 2$  with $\Psi(\alpha, 2)= \frac{2\alpha}{1-\alpha}<0$. Also, since $\binom{n-1}{i-1} (\frac{i-1}{n-1})^{i-1} (\frac{n-i}{n-1})^{n-i}\leq 1$ for all $i=1, \ldots, n$, we have 
  $\Psi(\alpha, n) \geq \frac{\alpha}{1-\alpha} \sum_{i=1}^n \log n=\frac{n\alpha}{1-\alpha}\log n. $ 
\end{proof}

\section{Kullback-Leibler  Information of Ranked Set Samples}\label{sec-kl}

In 1951 Kullback and Leiber introduced  a measure of information from the statistical  point of view by  comparing  two probability distributions associated with the same experiment. The Kullback-Leibler (KL) divergence is a measure of how different two probability distributions (over the same sample space) are.  The KL divergence for two random variables $X$ and $Y$ with  cdfs $F$ and $G$ and  pdfs $f$ and $g$, respectively, is given by
\begin{align}
\label{KL-1}
\K(X, Y)=\int f(t) \log \left(\frac{f(t)}{g(t)}\right) dt.
\end{align}
Using the same idea, we define the KL discrimination information between  $\X_{RSS}$ and $\X_{SRS}$ as follows:
\begin{align*}
\label{KL}
\K(\X_{SRS}, \X_{RSS})=\int_{\mathcal{X}^n} f(\x_{SRS}) \log\left(\frac{f(\x_{SRS})}{f(\x_{RSS})}\right) d\x_{SRS}.
\end{align*}
It is easy  to see that
\begin{eqnarray}
\label{KL-3}
\K(\X_{SRS}, \X_{RSS})= \sum_{i=1}^n \int f(x) \log\left( \frac{f(x) }{ f_{(i)}(x)}\right) dx  =\sum_{i=1}^n \K( X, X_{(i)}).
\end{eqnarray}
By substituting  the pdf of $ X_{(i)}$ in \eqref{KL-3}, we find
\begin{eqnarray}
 \label{KL-4}
\K(\X_{SRS}, \X_{RSS})
&=&-  \sum_{i=1}^n \int f(x) \log\left( \frac{f(x) }{ i\binom{n}{i} f(x) F^{i-1}(x) \bar{F}^{n-i} (x) }\right) dx\nonumber \\
&= &-\sum_{i=1}^n \int_{0}^1 \log\left\{ i\binom{n}{i}  u^{i-1} (1-u)^{n-i}\right\} du\nonumber\\
&=&- \sum_{i=1}^n  \log i\binom{n}{i} + n(n-1)\nonumber\\
&:=& d_n.
\end{eqnarray}
Note that  $\K(\X_{SRS}, \X_{RSS})$ is distribution-free, and
 $\{ d_n, n=1, 2, \ldots\}$ is a nondecreasing sequence of non-negative real values  for all $n\in\mathbb{N}$. That is, the KL information  between
the distribution of  SRS and  the distribution of RSS of the same size  increases
as  the set size $n$ increases. \\

\begin{remark}It is well known  that  the KL divergence is  non-symmetric and can not be considered as a distance metric. In our problem, note that 
\begin{align*}
\K(\X_{RSS}, \X_{SRS})
  =\sum_{i=1}^n \K(  X_{(i)}, X)
  =\sum_{i=1}^n \K( U_{(i)}, U)
  =-\sum_{i=1}^n {\rm H}(U_{(i)})=-k(n),
\end{align*}
where  $U$ has  uniform distribution. Various measures have been introduced in the literature generalizing this measure. For example, in order to have a distance metric,   the following  symmetric Kullback-Leibler  distance (KLD)  is proposed.
\begin{equation*}
\label{KL-1}
\text{KLD}(X, Y)= \K(X, Y)+\K(Y, X).
\end{equation*}
\end{remark}
\begin{lemma}
Suppose $\X_{SRS}$ is a SRS of size $n$  from $f (x)$ and let $\Y_{RSS}$ and $\Y_{SRS}$ be independent RSS  and SRS samples of the same size from another distribution  with pdf $g(x)$, respectively.  Then, $$\K(\X_{SRS}, \Y_{SRS})\leq \K( \X_{SRS}, \Y_{RSS}). $$
\end{lemma}
\begin{proof} To show the result, by the use of  the fact that   $g(x) =\frac1n\sum_{i=1}^n g_{(i)}(x)$, we have
\begin{eqnarray*}
\K( \X_{SRS}, \Y_{RSS})&=& \sum_{i=1}^n \int f(x) \log \left(\frac{f(x)}{ g_{(i)}(x)} \right) dx \\
&\geq& n\int f(x) \left\{ -\log\left(\sum_{i=1}^n\frac{ g_{(i)}(x) }{n\, f(x)} \right) \right\} dx \\
&=& n \int f(x) \log \left( \frac{f(x)}{g(x)}\right) dx \\&=& \K(\X_{SRS}, \Y_{SRS}),
\end{eqnarray*}
where the inequality is due to the convexity of $h(t)=-\log t$.
\end{proof}

\vskip 2mm
\noindent Now,  let  $\X^*_{RSS}=\{ X_{[i]}, i=1, \ldots, n\} $  be   an imperfect RSS of size $n$ from $f(x)$. Then,
\begin{eqnarray*}
\K( \X_{SRS}, \X^*_{RSS}) &=& \sum_{i=1}^n \int f(x) \log \left( \frac{f(x)}{f_{[i]}(x)}\right) dx\\&=& -\sum_{i=1}^n \int f(x) \log \left(\sum_{j=1}^n p_{i,j} j \binom{n}{j} F^{(j-1)}(x) \bar{F}^{(n-j)} (x) \right) dx  \\&=& -n \sum_{i=1}^n \int_0^1 \log \left( \sum_{j=1}^n p_{i,j} j\binom{n}{j} u^{j-1} (1-u)^{n-j}\right) du.  \end{eqnarray*}

\noindent Therefore, the KL discrimination information between the distribution of $\X_{SRS}$ and $\X^*_{RSS}$
is also distribution free and it is only a function of the set size $n$ and the ranking error  probabilities     $p_{i,j}= P(X_{[i]}= X_{(j)} )$.

\noindent In the following lemma, we show  that the KL information between the distribution of a  SRS and a perfect  RSS of the same size  is greater that the one with imperfect RSS.
\begin{lemma} Suppose $\X_{SRS}$ is a SRS of size $n$  from the pdf $f (x)$ and denote  $\X^*_{RSS}$ and $\X_{RSS}$ as independent   perfect and imperfect RSS data of the same size from $f$, respectively. Then, 
$$\K(\X_{SRS}, \X_{RSS}) \geq  \K (\X_{SRS}, \X^*_{RSS}).$$
\end{lemma}
\begin{proof} To show the result note that
\begin{eqnarray*}
 \K (\X_{SRS}, \X^*_{RSS} )
&=&  -\sum_{i=1}^n \int_0^1 \log \left( \sum_{j=1}^n p_{i,j} j\binom{n}{j} u^{j-1} (1-u)^{n-j}\right) du\\
&\leq & -\sum_{i=1}^n\sum_{j=1}^n p_{i,j} \int_0^1\log\left( j\binom{n}{j} u^{j-1} (1-u)^{n-j}\right)du\\
&=&-\sum_{j=1}^n    \int_0^1 \log\left( j\binom{n}{j} u^{j-1} (1-u)^{n-j}\right)du \\
&=&\K(\X_{SRS}, \X_{RSS}),\end{eqnarray*}
which completes the proof.
\end{proof}
\vskip 2mm

\noindent Another result which is of interest is to compare
$\K(\X_{RSS}, \Y_{RSS})$  with $\K(\X_{SRS}, \Y_{SRS})$. To this end,  we have
\begin{eqnarray*}
K(\X_{RSS}, \Y_{RSS})
&=& \sum_{i=1}^n \int f_{(i)}(x)  \left\{ \log\left(\frac{f(x)}{ g(x) }\right)  + \log \left(\frac{F^{(i-1)}(x) \bar{F}^{(n-i)}(x)}{G^{(i-1)}(x) \bar{G}^{(n-i)}(x)}\right)\right\} dx\\
&=& n  \int f(x) \log \frac{f(x)}{ g(x) } dx
 + \sum_{i=1}^n \int f(x)  i\binom{n}{i} F^{(i-1)}(x) \bar{F}^{(n-i)}(x) \log\left( \frac{F^{(i-1)}(x) \bar{F}^{(n-i)}(x)}{G^{(i-1)}(x) \bar{G}^{(n-i)}(x)}\right) dx\\
 &=& \K(\X_{SRS}, \Y_{SRS}) + A_n(F, G),
\end{eqnarray*}
where
$$ A_n(F,G)=  \sum_{i=1}^n \int _0^1 i\binom{n}{i} u^{i-1} (1-u) ^{n-i} \log\left(\frac{ u^{i-1} (1-u)^{n-i}}{ \{G(F^{-1}(u))\}^ {i-1}  \{\bar{G}(F^{-1}(u))\}^ {n-i} } \right) du.$$
Here again  $\K(\X_{SRS}, \Y_{SRS}) \leq \K(\X_{RSS}, \Y_{RSS})$ if $A_n(F, G)\geq 0$. Furthermore,  it is easy to show that
\begin{eqnarray*}
A_n(F, G) = -\frac{n(n-1)}{2}  - n(n-1) \int_0^1 \left\{ u\log G(F^{-1}(u)) + (1-u) \bar{G}(F^{-1}(u))\right\} du.
\end{eqnarray*}

Note that in this case  $A_n(F, G)$ depends on the parent distributions of $X$ and  $Y$ samples.

\begin{example}
 Suppose that $X$ and $Y$ have the exponential distributions  with parameters $\lambda_1$ and $\lambda_2$, and pdfs $f(x)= \lambda_1 e^{-\lambda_1 x}$ and $g(y)=\lambda_2 e^{-\lambda_2 y} $,  respectively. Then
\begin{eqnarray*}
A_n(F, G)& =& -\frac{n(n-1)}{2}  - \frac{n(n-1)}{3}\left(\frac{\lambda_2}{\lambda_1}\right)^2- n(n-1)\left\{ \sum_{i=1}^{\infty} \frac{1}{i(i+1)}-\frac{\lambda_2}{\lambda_1}\sum_{i=1}^{\infty}\frac{1}{i(i+2)}\right\}\\
&=&n(n-1)\left[\frac{\lambda_2}{\lambda_1}\left(\frac{1}{4}-\frac{\lambda_2}{3\lambda_1}\right)-\frac{3}{2} \right]<0.
\end{eqnarray*}
So, for  the exponential distributions   $\K(\X_{SRS}, \Y_{SRS}) > \K(\X_{RSS}, \Y_{RSS})$.
\end{example}

\section{Concluding Remarks}\label{sec-concluding}
In this paper, we have considered the information content of the  perfect and imperfect RSS data using the Shannon entropy, R\'enyi and Kullback-Leibler information measures. First, we have compared the Shannon entropy of a SRS data ($\X_{RSS}$) to the Shannon entropy of  perfect RSS ($\X_{RSS}$) and  imperfect  RSS ($\X_{RSS}^*$)  of the same size. In this case, we have established analytically that the Shannon entropies of $\X_{RSS}$ and $\X_{RSS}^*$ are less that the Shannon entropy of $\X_{SRS}$.	We also showed that the Shannon entropy of the RSS data will increase in the presence of the ranking error.  Next, we have established similar behaviour   under  the R\'enyi information when $0<\alpha<1$, while the results for the case were $\alpha>1$ remain unsolved. We conjectured that similar results hold for the case where $\alpha>1$ and provided  examples to support the conjecture.     Similar results are obtained under the Kullback-Leibler information measure. The results of this paper show  some  desirable properties of  ranked set sampling compared with the commonly used simple random sampling in the context of the information theory.

\section*{Acknowledgements}
Mohammad Jafari Jozani  acknowledges the research support of the Natural Sciences and Engineering Research Council of Canada.

\end{document}